\documentclass[11pt,reqno]{amsart}

\usepackage{amssymb,amscd, amsmath}
\usepackage{latexsym}
\usepackage{euscript}
\usepackage{hyperref}

\usepackage{geometry}
\usepackage{layout}
\newtheorem{theorem}{Theorem}
\newtheorem{lem}[theorem]{Lemma}
\newtheorem{cor}[theorem]{Corollary}
\newtheorem{prop}[theorem]{Proposition}

\def\scalprod#1#2{({#1}\, |\, {#2})}
\newcommand{\beeq}{\begin{eqnarray*}}
\newcommand{\eneq}{\end{eqnarray*}}

\DeclareMathOperator{\coker}{coker}
\def\SS{\EuScript S}

\newcommand{\quot}[2]{{#1/#2}}

\author{Christine Bachoc}
\address{Institut de Math\'ematiques de Bordeaux, UMR 5251, universit\'e de Bordeaux, 351 cours de la Lib\'eration, 33400 Talence, France.}
\email{Christine.Bachoc@math.u-bordeaux.fr}
\author{Oriol Serra}
\address{Department of Mathematics, Universitat Polit\`ecnica de Catalunya  and Barcelona Graduate School of Mathematics, Barcelona, Spain.}
\email{oriol.serra@upc.edu}
\author{Gilles Z\'emor}
\address{Institut de Math\'ematiques de Bordeaux, UMR 5251, universit\'e de Bordeaux, 351 cours de la Lib\'eration, 33400 Talence, France.}
\email{zemor@math.u-bordeaux.fr}

\title[]{Revisiting Kneser's Theorem for Field Extensions}
\date{}

\begin{document}

\begin{abstract}
  A Theorem of Hou, Leung and Xiang 
  generalised Kneser's addition Theorem to field extensions.
  This theorem was known to be valid only in separable extensions, and
  it was a conjecture of Hou that it should be valid for all extensions. 
  We give an alternative proof of the theorem that also holds in the
  non-separable case, thus solving Hou's conjecture. This result is a
  consequence of a strengthening of Hou et al.'s theorem that is
  inspired by an addition theorem of
  Balandraud and is obtained by combinatorial methods transposed and
  adapted to the extension
  field setting.
%
\end{abstract}

\maketitle

\section{Introduction}
\let\thefootnote\relax\footnote{{\it Mathematics Subject Classification:} 11P70.\\
  Financial support for this research was provided by the ``Investments for the future''  Programme IdEx Bordeaux CPU (ANR-10-IDEX- 03-02) and  the Spanish Ministerio de Econom\'ia y Competitividad under project MTM2014-54745-P}

Let $G$ be an abelian group and let $S$ and $T$ be finite
subsets of $G$. Denote by $S+T$ the set defined by $\{s+t : s\in S,t\in T\}$.
Sets $S+T$ are often referred to as {\em sumsets}. Kneser's classical addition
Theorem \cite{kneser} states that sufficiently small sumsets are periodic, meaning they are stabilized by non-zero group elements.

\begin{theorem}[Kneser]\label{thm:knesergroup}
  Let $G$ be an abelian group and let $S$ be non-empty, finite subsets of~$G$.
  Then one of the following holds:
  \begin{itemize}
  \item $|S+T|\geq |S|+|T|-1$,
  \item there exists a subgroup $H\neq\{0\}$ of $G$ such that
        $S+T+H=S+T$.
  \end{itemize}
\end{theorem}

Kneser's Theorem is more usually stated as $|S+T|\geq |S|+|T|-|H(S+T)|$
where $H(S+T)=\{x\in G : x+S+T=S+T\}$ denotes the stabilizer of $S+T$. However,
this more precise inequality is easily derived
from Theorem~\ref{thm:knesergroup} which we chose to formulate in this way because it is better suited to
the discussion of the strengthened versions that will follow.
Kneser's Theorem is one of the founding theorems of additive combinatorics and
has many applications to this field, and more generally to situations
where statements on the structure of sumsets are useful, see \cite{Nat,tv} for
example.

Consider now the context of field extensions.
Let $F$ be a field and let $L$ be an extension field of
$F$. If $S$ and $T$ are $F$--vector subspaces of $L$, we shall denote
by $ST$ the $F$-linear span of the set of products $st$, $s\in S$,
$t\in T$.

The following Theorem was obtained by Hou, Leung and Xiang
\cite{hou2002}, as a transposition to field extensions 
of Kneser's classical addition Theorem. 

\begin{theorem}[Hou, Leung and Xiang] \label{thm:kneser}
Let $F$ be a field, $L/F$ a field extension, and let $S$ and $T$ be $F$-subvectorspaces  of $L$ of finite dimension. 
Suppose that every algebraic element in $L$ is separable over
$F$. Then
one of the following holds:
\begin{itemize}
\item $\dim ST\ge \dim S+\dim T-1,$
\item there exists a subfield $K$, $F\subsetneq K\subset L$, such that
$STK=ST$.
\end{itemize}
\end{theorem}

The conclusion of the Theorem of Hou et al. given in \cite{hou2002} is
that we have
$$
\dim ST\ge \dim S+\dim T-\dim H(ST),
$$
where $H(ST)=\{x\in L: xST\subset ST\}$ denotes the stabilizer of $ST$ in
$L$.
As is explained in \cite{hou2002}, the above formulation is easily
seen to be equivalent to the conclusion of Theorem~\ref{thm:kneser}. 
Hou et al.'s Theorem \ref{thm:kneser} can also
be seen as a generalisation of the original additive Theorem of Kneser since
the latter can be recovered from the former.

Theorem \ref{thm:kneser} was initially motivated by a problem on
difference sets \cite{hou2002}, but has since become part of a wider effort
to transpose some classical theorems of additive combinatorics to a
linear algebra or extension field setting. In particular Eliahou and
Lecouvey \cite{eliahou2009} obtained linear analogues of some
classical additive theorems including theorems of Olson \cite{olson}
and Kemperman \cite{kemperman} in nonabelian groups. Lecouvey
\cite{lecouvey2011} pursued this direction by obtaining, among other
extensions,  linear versions of the Pl\"unecke--Ruzsa \cite{ruzsa}
inequalities. The present authors recently derived a linear analogue of
Vosper's Theorem in \cite{bachoc}.
Somewhat more generally, additive combinatorics have had some
spectacular successes by lifting purely additive problems into various
algebras where the additional structure has provided the key to the
original problems, e.g. \cite{DSH94,karolyi}. This provides in part additional
motivation for linear extensions of classical addition theorems.

Going back to Hou et al.'s Theorem~\ref{thm:kneser}, a natural question is
whether the separability assumption in Theorem~\ref{thm:kneser} is
actually necessary. Hou makes an attempt in \cite{hou2007} to work at
Theorem~\ref{thm:kneser} without the separability assumption, but only
manages a partial result where the involved spaces are assumed to have
small dimension. Hou goes on to conjecture \cite{hou2007} that 
Theorem~\ref{thm:kneser} always holds, i.e., holds without the
separability assumption. 
Recently, Beck and Lecouvey~\cite{bl15} extended
Theorem~\ref{thm:kneser} to algebras other than a field extension over
$F$, but again, their approach
breaks down when the algebra contains an infinity of subalgebras, so that the
case of non-separable field extensions is not covered.

In the present work, we prove Hou's conjecture and remove the
separability assumption in Theorem~\ref{thm:kneser}. We actually prove the
stronger statement below.

\begin{theorem}\label{thm:kneser_onesided}
 Let $L/F$ be a field extension, and let $S\subset L$ be an
 $F$-subspace of $L$ of finite positive dimension. Then 
 \begin{itemize}
 \item either for every finite dimensional subspace $T$ of $L$ we have
  $$\dim ST\geq\dim S+\dim T-1,$$
 \item or there exists a subfield
 $K$ of $L$, $F\subsetneq K\subset L$, such that for every
 finite-dimensional subspace
 $T$ of $L$ satisfying
  $$\dim ST < \dim S+\dim T-1,$$
 we have $STK=ST$.
 \end{itemize}
\end{theorem}

Besides the removal of the separability condition, the additional
strength of 
Theorem \ref{thm:kneser_onesided} with respect to Theorem
\ref{thm:kneser} lies in the fact that the subfield $K$ that
stabilises $ST$
 seems to depend on both spaces $S$ and $T$ in Theorem~\ref{thm:kneser}  but actually can be seen to depend only on one of the factors in Theorem~\ref{thm:kneser_onesided}.
Theorem~\ref{thm:kneser_onesided} is a transposition to the extension
field setting of a theorem of Balandraud \cite{balandraud} which is a
similarly stronger form of Kneser's Addition Theorem and can
be stated as:

\begin{theorem}\label{thm:balandraud}
 Let $G$ be an abelian group, and let $S\subset G$ be a finite subset of
 elements of $G$. Then
 \begin{itemize}
 \item Either for every finite subset $T$ of $G$ we have
  $$|S+T|\geq |S|+|T|-1,$$
 \item or there exists a subgroup $H\neq\{0\}$ of $G$ such that, for every
   finite subset $T$ of $G$ satisfying
 $$|S+T| < |S|+|T|-1,$$
 we have $S+T+H=S+T$.
 \end{itemize}
\end{theorem}

Balandraud proved his theorem through an in-depth study of a
phenom\-enon that he called saturation. For a given set $S$, a set $T$
is saturated with respect to $S$ if there does not exist a set $T'$
strictly containing $T$ such that the sums $S+T$ and $S+T'$ are equal.
He showed that when $T$ is a subset of smallest cardinality among all
saturated subsets for which the quantity $|S+T|-|T|$ is a given
constant, then $T$ must be a coset of some subgroup of
$G$. Furthermore, the subgroups that appear in this way form a chain
of nested subgroups, and the smallest non-trivial subgroup of this
chain is the subgroup $H$ of Theorem~\ref{thm:balandraud}.

The techniques used by Balandraud are very combinatorial in nature and are inspired by Hamidoune's isoperimetric (or
atomic) method in additive combinatorics \cite{hamidoune,YOH13}.
In the present paper we prove Theorem~\ref{thm:kneser_onesided}
by exporting Balandraud's approach to the extension field setting.
This can also be seen as a follow-up to the linear isoperimetric
method initiated in Section~3 of~\cite{bachoc}.
We note that this strategy deviates significantly from Hou et al.'s
approach in \cite{hou2002} which relied on a linear variant of the
additive $e$-transform and required crucially the separability of
the field extensions.

We also note that Theorem~\ref{thm:kneser_onesided} can be seen
as a generalisation of Balandraud's Theorem~\ref{thm:balandraud} in
groups, since the latter may be derived from the former by exactly the
same Galois group argument as that of \cite[Section 3]{hou2002}.

The paper is organised as follows.
Section~\ref{sec:con} is devoted to setting up some basic tools and
deriving the combinatorics of saturation.
Section~\ref{sec:main} introduces {\em kernels}, which are finite
dimensional subspaces of minimum dimension among finite-dimensional
subspaces $T$ for which $\dim ST - \dim T$ is a fixed integer less than
$\dim S -1$. A structural theory of kernels is derived, whose core
features are collected in Theorem~\ref{thm:linbal}. Its proof is
broken up into intermediate results through
Propositions \ref{claim5}, \ref{claim6} and \ref{claim7}.
Finally, Section~\ref{sec:proofmain} derives the proof of 
Theorem~\ref{thm:kneser_onesided} from the structure of kernels and
concludes the paper.

\section{Preliminary definitions and properties}\label{sec:con}

We assume that $L/F$ is a field extension and that $S$ is a
finite-dimensional $F$-subspace of $L$ such that $1\in S$.
We suppose furthermore that $F(S)=L$, where
$F(S)$ denotes the subfield of $L$ generated by $S$.

\subsection{Boundary operator and submodularity} 
For every  subspace $X$ of $L$  we define
$$
\partial_S X=\dim\, \quot{XS}{X}
$$
the increment of dimension of $X$ when multiplied by $S$. 
We omit the
subscript $S$ in $\partial_S$ whenever $S$ is clear from the
context. 
Note that we may have $\partial{X}=\infty$
and that when $X$ and $S$ are finite-dimensional $\partial{X}=\dim XS -\dim X$.
The essential property of the ``boundary'' operator
$\partial$
is the submodularity relation:

\begin{prop}\label{prop:subm} Let $X, Y$ be  subspaces of $L$. We have
$$
\partial (X+Y)+\partial (X\cap Y)\le \partial X +\partial Y.
$$
\end{prop}

A short proof of Proposition~\ref{prop:subm} is given in \cite{bachoc}
when $L$ is finite-dimensional over $F$. In the general case, we invoke
the following Lemma:

\begin{lem}\label{lem:snake}
  Let $A,B,A'$ and $B'$ be subspaces of some ambient vector space $E$ such
  that $A\subset A'$ and $B\subset B'$. There is an exact sequence of
  vector spaces
$$0\rightarrow \quot{(A'\cap B')}{(A\cap B)} \rightarrow
\quot{A'}{A}\times \quot{B'}{B} \rightarrow \quot{(A'+B')}{(A+B)} \rightarrow 0.$$
\end{lem}

\begin{proof}
We may identify the subspace $A\cap B$ with the subspace of $A\times
B$ consisting of the elements $(x,-x)$, $x\in A\cap B$. With the
similar identification for $A'\cap B '$, we get the isomorphisms
\begin{equation}
  \label{eq:A+B,A'+B'}
  \quot{(A\times B)}{(A\cap B)} \xrightarrow{\sim} A+B\quad\text{and}\quad
\quot{(A'\times B')}{(A'\cap B')} \xrightarrow{\sim} A'+B'
\end{equation}
and the following commutative diagram with the rows being
exact and $\gamma$ corresponding to the natural mapping of $A+B$ into $A'+B'$.
  $$\begin{CD}
  @.  A\cap B   @>>> A\times B   @>>> \quot{(A\times B)}{(A\cap B)}@>>> 0\\
  @.  @VV{\alpha}V           @VV{\beta}V             @VV{\gamma}V                             @.\\
  0 @>>>  A'\cap B' @>>> A'\times B' @>>> \quot{(A'\times B')}{(A'\cap B')}@.
  \end{CD}
  $$
The snake lemma (Lang, \cite[Ch 3, Section 9]{lang}) therefore gives the exact sequence
$$0\rightarrow\coker\alpha \rightarrow\coker\beta
\rightarrow\coker\gamma \rightarrow 0,$$
which yields the result after identification of $\quot{A'\times
  B'}{A\times B}$ with $\quot{A'}{A}\times \quot{B'}{B}$ and
the identifications \eqref{eq:A+B,A'+B'}. 
\end{proof}

Lemma~\ref{lem:snake} immediately gives:
\begin{cor}\label{cor:snake}
 If $A'/A$ and $B'/B$ have finite dimension, then
 $$\dim\,\quot{(A'+B')}{(A+B)} = \dim\,\quot{A'}{A} +
 \dim\,\quot{B'}{B} - \dim\, \quot{(A'\cap B')}{(A\cap B)}.$$
\end{cor}

\noindent
{\it Proof of Proposition~\ref{prop:subm}}.
  If $\partial{X}$ or $\partial{Y}$ is $\infty$, there is nothing to
  prove, so we may set $X'=XS$ and $Y'=YS$ and suppose that $X'/X$ and
  $Y'/Y$ are finite-dimensional. We have $(X+Y)S\subset X'+Y'$ and
  $(X\cap Y)S\subset X'\cap Y'$, whence
   $$\partial{(X+Y)}+\partial{(X\cap Y)}\leq \dim\,\quot{(X'+Y')}{(X+Y)} +\dim\,\quot{(X'\cap Y')}{(X\cap Y)},$$
and the conclusion follows from Corollary~\ref{cor:snake}.

\subsection{Duality} 
Recall that every non-zero linear form $\sigma:L\to F$ induces a  nondegenerate  symmetric bilinear (Frobenius) form defined as $\scalprod{x}{y}_{\sigma}=\sigma(xy)$, with  the property:
\begin{equation}
  \label{eq:frobenius}
  \scalprod{xy}{z}_{\sigma}=\scalprod{x}{yz}_{\sigma} \; \mbox{ for all } x,y,z\in L.
\end{equation}
Fix such a bilinear form $(\cdot|\cdot)$. For a subspace $X$, we set
$$
X^\perp = \{y\in L:\; \forall x\in X, \scalprod{x}{y}=0\}.
$$
We call the {\it dual subspace} of  the subspace $X$ the subspace
$$
X^*=(XS)^\perp.
$$

We will use the notation $X^{**} =(X^*)^*$ and $X^{***} =(X^{**})^*=(X^*)^{**}$.

We shall require the following lemma which is a straightforward
consequence of Bourbaki \cite[Ch. 9, \S 1, n. 6,
  Proposition 4]{bourbaki}:
  \begin{lem}\label{lem:bourbaki}
    If $A$ and $B$ are subspaces such that 
    the quotient $\quot{A}{(B^\perp\cap A)}$ is finite
  dimensional, then $\dim\quot{A}{(B^\perp\cap A)} = \dim
  \quot{B}{(A^\perp\cap B)}$. 
  \end{lem}

The following elementary properties hold for subspaces and their duals:

\begin{lem}\label{lem:dual} For every  $F$-subspace $X$ of $L$, we have
\begin{enumerate}
\item[\rm{(i)}] $X\subset X^{**}$.
\item[\rm{(ii)}] $X^*=X^{***}$
\item[\rm{(iii)}] $\partial X^*\le \partial X$.
\end{enumerate}
\end{lem}

\begin{proof}
\hspace{1cm}

  \begin{enumerate}
  \item[\rm{(i)}] Let $x\in X$ and let $x^*\in X^*$ and $s\in S$. By definition of
    $X^*$ we have $\scalprod{xs}{x^*}=0$, so that
    $\scalprod{x}{x^*s}=0$. Therefore $x\in (X^*S)^\perp=X^{**}$.
\item[\rm{(ii)}] Applying \rm{(i)}, we have $X^*\subset
  (X^*)^{**}$. Also, if $Y\subset Z$, then $Z^*\subset
  Y^*$, which yields $(X^{**})^*\subset X^*$.
  \item[\rm{(iii)}] 
  If $\partial X=\infty$ there is nothing to prove, so assume
  $\partial X<\infty$. From \rm{(i)} and $X\subset XS$, we have
  $X\subset (X^{**}\cap XS)$. Hence
  $$\dim\quot{XS}{(X^{**}\cap XS)}\leq
  \dim \quot{XS}{X}<\infty.$$
  Applying Lemma~\ref{lem:bourbaki}, we therefore have
  \begin{align*}
    \partial X^*&=\dim \quot{X^*S}{X^*}=\dim
    \quot{X^*S}{(XS)^\perp}=\dim \quot{XS}{((X^*S)^\perp\cap XS)}\\
&=\dim\quot{XS}{(X^{**}\cap XS)} \leq \dim \quot{XS}{X}=\partial X. 
  \end{align*}
  \end{enumerate}
\end{proof}

One would expect the stronger properties $X^{**}=X$ and $\partial
X^*=\partial X$ to hold. Unfortunately this is not true for all
subspaces, only for those who are \emph{saturated}, a notion that we
introduce below.

\subsection{Saturated spaces} For a subspace $X$ let
us define the subspace $\tilde{X}$ to be the set of all $x\in L$ such
that
$$xS\subset XS.$$

Clearly we have $X\subset \tilde{X}$, $\tilde{X}S=XS$,  and $\partial \tilde{X}\leq \partial X$.
We remark also that $\tilde{X}\subset \tilde{X}S=XS$ implies that
whenever $X$ is finite-dimensional, so is $\tilde{X}$.

A  subspace $X$ is said to be {\it saturated} if
$$
\tilde{X}=X.
$$
\begin{lem}\label{lem:satdual} For every $F$-subspace $X$ of $L$, the following hold:
\begin{enumerate}
\item[\rm{(i)}]  $X^*$ is saturated.
\item[\rm{(ii)}] If $X$ is finite-dimensional then $X^{**}=\tilde{X}$. In
particular a finite-dimensional subspace $X$ is saturated if and only if $X=X^{**}$.
\end{enumerate}
\end{lem}

\begin{proof} \hspace{1cm}
\begin{enumerate}
\item[\rm{(i)}] Let $y\in L$ be such that $yS\subset X^*S$, and let us
  prove that $y\in X^*=(XS)^\perp$. Since $ys\in X^*S$, we have
$ys=\sum x^*_is_i$ where $x^*_i\in X^*$ and $s_i\in S$. Therefore,
for any $x\in X$, $s\in S$, we have
\begin{equation*}
\scalprod{y}{xs}=\scalprod{ys}{x}=\sum_i \scalprod{x^*_is_i}{x}=\sum_i
\scalprod{x^*_i}{s_ix}=0,
\end{equation*}
which means that $y\in X^*$.
\item[\rm{(ii)}] We recall that, for a finite-dimensional subspace  $A$, we have
  $(A^\perp)^\perp = A$ (Bourbaki, \cite[Ch. 9, \S 1, n. 6, cor. 1]{bourbaki}).
The assertion follows from:
  \begin{eqnarray*}
    y\in X^{**} & \Leftrightarrow & y\in (X^*S)^\perp\\
                 & \Leftrightarrow & \forall x^*\in X^*, \forall s\in
                                     S, \; \scalprod{y}{x^*s}=0\\
                 & \Leftrightarrow & \forall x^*\in X^*, \forall s\in
                                     S, \; \scalprod{ys}{x^*}=0\\
                 & \Leftrightarrow & yS\subset ((XS)^\perp)^\perp =
                                     XS\\
                 & \Leftrightarrow & y\in\tilde{X}.
  \end{eqnarray*}
\end{enumerate}
\end{proof}

We denote by $\SS$ the family of saturated finite-dimensional 
subspaces $X$ of $L$ together with their duals $X^*$.
We make the remark that, applying Lemma~\ref{lem:bourbaki} with $A=XS$ and $B=L$, the dual of a finite dimensional space has finite co-dimension (where the co-dimension of a space $A$ is defined as
$\dim L/A$). In particular, elements of $\SS$ have either finite dimension or finite co-dimension.

The next lemma
summarizes the properties of the elements of $\SS$ that we will need in
the proof of Theorem \ref{thm:kneser_onesided}.

\begin{lem}\label{lem:SS}
For every $X\in \SS$, $Y\in \SS$,
\begin{enumerate}
\item[\rm{(i)}] $X$ is saturated and $X^*\in \SS$.
\item[\rm{(ii)}] $X^{**}=X$.
\item[\rm{(iii)}] $\partial X=\partial X^*$.
\item[\rm{(iv)}] $X\cap Y\in \SS$.
\item[\rm{(v)}] $(X+Y)^{**}\in \SS$.
\end{enumerate}
\end{lem}

\begin{proof}
\hspace{1cm} 
\begin{enumerate}
\item[\rm{(i)} \rm{(ii)}] $X$ is saturated by Lemma \ref{lem:satdual}\rm{(i)}. 
If $X\in \SS$ has finite dimension, $X^*$ belongs to $\SS$ by definition, and $X^{**}=X$ by Lemma \ref{lem:satdual}\rm{(ii)}. 
Otherwise, $X=X_1^*$ where $X_1$ is saturated and of finite dimension, and $X^*=X_1^{**}=X_1$ by Lemma \ref{lem:satdual}\rm{(ii)}, so $X^*$ 
belongs to $\SS$ and $X^{**}=X_1^*=X$.

\item[\rm{(iii)}] From Lemma \ref{lem:dual}\rm{(iii)}, we have $\partial X^*\leq \partial X$. Additionally, applying this inequality to $X^*$ and combining with $X^{**}=X$ leads to 
$\partial X^*= \partial X$.

\item[\rm{(iv)}] We have
\begin{equation*}
X^*\cap Y^*=(XS)^\perp \cap (YS)^\perp=(XS+YS)^\perp=((X+Y)S)^\perp=(X+Y)^*.
\end{equation*}
In particular, if $X$ and $Y$ belong to $\SS$, $X\cap Y=X^{**}\cap Y^{**}=(X^*+Y^*)^*$ so, by 
Lemma \ref{lem:satdual}\rm{(i)}, $X\cap Y$ is saturated. If, moreover, $X$ or $Y$ is of finite dimension, we can conclude that $X\cap Y\in \SS$.
Otherwise, $X=X_1^*$ and $Y=Y_1^*$, where $X_1$ and $Y_1$ are both of finite dimension, and $X\cap Y=(X_1+Y_1)^*$. We remark that 
$(X_1+Y_1)^*=(X_1+Y_1)^{***}=(\widetilde{X_1+Y_1})^*$, applying Lemma \ref{lem:dual}\rm{(ii)} and Lemma \ref{lem:satdual}\rm{(ii)}, so $X\cap Y\in \SS$.

\item[\rm{(v)}] If the dimensions of $X$ and $Y$ are finite, then, by Lemma \ref{lem:satdual}\rm{(ii)}, $(X+Y)^{**}=\widetilde{X+Y}$ belongs to $\SS$.
Otherwise, without loss of generality we may assume that $X=X_1^*$ where $X_1$ is saturated and of finite dimension. Let $Z:=(X+Y)^*$. 
By Lemma \ref{lem:satdual}\rm{(i)}, $Z$ is saturated, and by Lemma \ref{lem:satdual}\rm{(ii)}, $Z\subset X^*=X_1^{**}=X_1$, so $Z$ is of finite dimension and we can conclude
that $Z^*=(X+Y)^{**}$ belongs to $\SS$. 
\end{enumerate}
\end{proof}

In the proof of Theorem \ref{thm:kneser_onesided}, we will apply many
times the submodularity inequality of Proposition \ref{prop:subm} to
certain subspaces $X$ and $Y$ belonging to $\SS$. We will have  $X\cap Y\in
\SS$  (Lemma \ref{lem:SS} \rm{(iv)}), but we will have to deal with
the issue that in general $X+Y\notin \SS$.
Lemma \ref{lem:SS} \rm{(v)} will allow us to replace $X+Y$ by the larger
$(X+Y)^{**}$ since $\partial (X+Y)^{**}\leq \partial (X+Y)$. The
following Lemma will be used several times in order to ensure that 
$(X+Y)^{**}\neq L$ holds and that we do not have $\partial (X+Y)^{**}=0$.

\begin{lem}\label{lem:sum}
Let $X$ and $Y$ be subspaces of $L$ such that  $\dim X<\infty$,
$\dim XS\leq\dim X +\dim S -1$,  $\dim(X\cap Y)\geq 1$ and $\dim X
\leq \dim Y^*$.
Then
\begin{equation*}
(X+Y)^{**}\neq L.
\end{equation*}
\end{lem}

\begin{proof}
We will prove that $(X+Y)^*=(X+Y)^{***}\neq L^*=\{0\}$. 
Since $Y\subset (X+Y)$, we have $(X+Y)^*\subset Y^*$. 
We will show that $\dim
\quot{Y^*}{(X+Y)^*}$ is finite and less than $\dim(Y^*)$, which will
imply $(X+Y)^*\neq \{0\}$. Note that 
$$YS\subset ((YS)^\perp)^\perp\cap (X+Y)S\subset (X+Y)S\subset XS+YS.$$
Therefore
$$\dim \quot{(X+Y)S}{\left(((YS)^\perp)^\perp\cap (X+Y)S\right)}\leq \dim\quot{(XS+YS)}{YS}.$$
The right-hand side is finite, whence also the left-hand side, which,
by Lemma~\ref{lem:bourbaki}, equals $\dim \quot{Y^*}{(X+Y)^*}$.
We therefore have

\begin{align*}
\dim \quot{Y^*}{(X+Y)^*}&\leq \dim\quot{(XS+YS)}{YS}=\dim \quot{XS}{(XS\cap
  YS)}\\
&=\dim XS-\dim(XS\cap YS).
\end{align*}
From the hypothesis we have $\dim XS\leq \dim X+\dim S-1$ and from
$\dim(X\cap Y)\geq 1$ we have $\dim(XS\cap YS)\geq \dim S$. Hence
$$\dim \quot{Y^*}{(X+Y)^*} \leq \dim X+\dim S-1 -\dim S=\dim X-1 <\dim
Y^*.$$
\end{proof}

When trying to prove that a saturated subspace $X$ has a non-trivial
stabilizer, it will be useful to consider its dual subspace instead.
The last Lemma of this section states that an element
stabilizes a saturated subspace if and only if it stabilizes
its dual subspace.

\begin{lem}\label{lem:subfield}
If $X\in \SS$ and $k\in L$, then $kX\subset X$ if and only if $kX^*\subset X^*$.
\end{lem}

\begin{proof}
For $x^*\in X^*$, $x\in X$ and $s\in S$, we have
\begin{equation*}
\scalprod{kx^*}{xs}=\scalprod{x^*}{kxs}
\end{equation*}
from which we get that if $kX\subset X$ then $\scalprod{kx^*}{xs}=0$ for every $x,x^*,s$, whence $kx^*$ is in $X^*$. Therefore $kX^*\subset X^*$.

if $kX^*\subset X^*$ then we have just proved
that $kX^{**}\subset X^{**}$, and Lemma \ref{lem:SS}
\rm{(ii)} gives the desired conclusion.
\end{proof}

We remark that if $X$ is finite dimensional and if $kX\subset X$ for some non-zero $k$, then
$k$ can only be of finite degree over $F$, and we have $k^{-1}X\subset X$,
(and therefore $X=kX$). The stabilizer $H(X)=\{k\in L, kX\subset X\}$ is a field in this case. Lemma~\ref{lem:subfield} implies in particular that
stabilizers of spaces of $\SS$ are subfields of $L$. Summarizing:

\begin{cor}\label{cor:subfield}
  If $X\in \SS$, then the stabilizer $H(X)$ is a subfield satisfying
  $H(X)=H(X^*)$.
\end{cor}

\section{Structure of cells and kernels of a subspace}\label{sec:main}

We assume, like in the previous section, that $S$ is a finite dimensional $F$-subspace of $L$ containing $1$, and that $L=F(S)$. 
We will moreover  assume that there does not exist a field $K$, $F\subsetneq K\subset L$, such that $KS=S$ (in other words $H(S)=F$). Note that
when such a $K$ exists, the conclusion of Theorem~\ref{thm:kneser_onesided} holds trivially.
With the objective of working towards a proof of Theorem~\ref{thm:kneser_onesided},
we will also assume that there exists a non-zero finite dimensional subspace $T\subset L$ such
that $ST\neq L$ and 
$$
\dim ST<\dim S+\dim T-1.
$$
Equivalently, $\partial T<\dim S - 1$. 

Let
$$
\Lambda=\{\partial (X): X\in \SS  \}.
$$
We denote the elements of $\Lambda$ by
$$
\Lambda=\{ 0=\lambda_0<\lambda_1<\lambda_2<\cdots \}
$$
and by
$$
\SS_i=\{ X\in \SS: \partial X=\lambda_i\}.
$$
Spaces belonging to a set $\SS_i$ will be called $i$--{\em cells}. By
Lemma \ref{lem:SS} \rm{(i) (iii)}, the dual of an $i$--cell is an $i$-cell.
An $i$--cell of smallest dimension  will be said to be an  $i$--{\em  kernel}.
We note that $i$--kernels are always of finite dimension, because the dual
of an infinite-dimensional $i$--cell is an $i$--cell and must have finite dimension. 

Suppose $X$ is finite-dimensional and $SX=X$. Then either $X=\{0\}$ or
$S$ must be a field, so that $S=F(S)=L$. From this we get that
$\SS_0=\{\{0\},L\}$. From our assumption on the existence of $T$, we get $\lambda_1<\dim S-1$. 
Let $n$ be the largest integer such that
$\lambda_n<\dim S$. We note that we have $\lambda_n=\dim S -1$ and
that $F$ is an $n$--kernel since $\partial F = \dim S -1$ and $F$ is
saturated, otherwise $\tilde{F}$ would contradict our assumption on the
non-existence of a field $K\neq F$ such that $KS=S$.

If $N$ is an $i$--kernel, then clearly so is $xN$ for any non-zero
$x\in L$. Therefore, when an $i$--kernel exists, there exists in
particular an
$i$--kernel containing $F$. Let $F_1,F_2,\ldots ,F_n$ be $1,2,\ldots
,n$--kernels containing $F$, which implies $F_n=F$ by the remark just above.

Our core result is the following theorem.

\begin{theorem}\label{thm:linbal} We have
  $$F_1\supset F_{2}\supset \cdots \supset F_n.$$
  Furthermore the $F_i$ are all subfields of $L$, and 
  every  space $X\in \SS_i$  is stabilized by $F_i$.
\end{theorem}

Note that this last statement implies in particular that the
$i$-kernel containing $F$ is unique. 

We shall prove Theorem~\ref{thm:linbal} in several steps.

First we prove the result for~$F_1$. This simple case illustrates
the general methodology that consists in intersecting the cell under
study $X$ with some other cell $Y$, and applying
the submodularity relation of Proposition~\ref{prop:subm}. The goal
is to prove that the intersection $X\cap Y$ is either $X$ or $Y$ by
arguing that one of the two-cells is a kernel, and that it has minimum
dimension among cells with a given boundary. For this one needs
to bound from above the boundary of the intersection $X\cap Y$,
which is achieved through Proposition~\ref{prop:subm} and 
a lower bound on the boundary of the sum $X+Y$. Most of the technicalities
go into deriving these lower bounds.

\begin{prop}\label{prop:F1}
  $F_1$ is a subfield of $L$ and any $1$--cell $X$ satisfies $XF_1=X$.
\end{prop}
\begin{proof}
  Let $X$ be a $1$--cell and let $x$ be a non-zero vector of $X$, so
  that $X$ has a non-zero intersection with $xF_1$.
By submodularity we have
$$
\partial (xF_1+ X)^{**}+\partial (xF_1\cap X)\le \partial (xF_1+ X)+\partial (xF_1\cap X)\le 2\lambda_1.
$$
By Lemma~\ref{lem:SS} \rm{(iv)} we have $xF_1\cap X\in \SS$. 
Since $xF_1\cap X$ is non-zero and not equal to $L$ (because $F_1\neq L$),
$xF_1\cap X\in \SS_k$ for some $k\ge 1$.
Since $xF_1$ is a $1$--kernel, and since $X^*$ is also a $1$--cell, we have $\dim xF_1\leq \dim X^*$ and since kernels are finite-dimensional, Lemma~\ref{lem:sum} implies $(xF_1+ X)^{**}\in \SS_\ell$ for some $\ell\ge 1$.
It  follows that $k=\ell=1$. Therefore, $xF_1\cap X$ is a $1$--cell and, by the minimality of the
dimension of $1$--kernels, we have $xF_1\subset X$. Since this holds for an
arbitrary $x\in X$,  we have proved $XF_1=X$. Applying this
to $X=F_1$ we obtain that $F_1$ is a subfield of $L$.
\end{proof}

Let $J$ be the set of positive integers $j\in [1,\ldots ,n]$ satisfying
the conditions
\begin{itemize}
\item
$F_1\supset \cdots \supset F_{j-1}\supset F_{j}$
\item $F_{j}$ is a subfield of $L$ 
\item any $j$--cell is stabilized by $F_{j}$. 
\end{itemize}
Proposition~\ref{prop:F1} tells us that $1\in J$, so that $J\neq\emptyset$.
The proof of Theorem~\ref{thm:linbal} will be complete if we can show that $J$ equals the whole interval $[1,n]$. We therefore assume by contradiction that $\overline{J}=[1,n]\setminus J\neq\emptyset$ and define $i$ to be the smallest integer in $\overline{J}$. We then proceed to show that the integer $i$ also satisfies the above three conditions, contradicting $i\not\in J$.
Specifically we shall prove that $F_i\subset F_{i-1}$ (Proposition~\ref{claim5}), 
that $F_i$ is also a subfield (Proposition~\ref{claim6}) and that $F_i$ stabilizes every
$i$--cell (Proposition~\ref{claim7}).

\begin{lem}\label{claim1}
  No $i$--cell $X$ is stabilized by $F_{i-1}$.
\end{lem}

\begin{proof} Suppose $F_{i-1}X=X$. Then $X$ and $SX$
 are $F_{i-1}$-vector spaces and $\lambda_{i}=\dim \quot{XS}{X}$
 is a multiple of $\dim F_{i-1}$. The quantity
 $\lambda_{i-1}=\dim F_{i-1}S - \dim F_{i-1}$ is also a multiple of
 $\dim F_{i-1}$, and since $\lambda_i>\lambda_{i-1}$,
 $$
 \lambda_i\ge \lambda_{i-1}+\dim F_{i-1}= \dim F_{i-1}S\ge \dim S,
 $$
 contradicting $\lambda_i<\dim S$. 
\end{proof}

\begin{lem}\label{lem:referee*}
  Let $1\leq j \leq i-1$, and suppose $\dim F_i\leq \dim F_j$. Then, for any $x\in F_i$ we have $(xF_j+F_i)^{**}\neq F_{j-1}$, where we adopt the convention $F_0=L$. In particular if  $\dim F_i\leq \dim F_1$, then $(xF_1+F_i)^{**}\neq L$.
\end{lem}

\begin{proof}
  Let $x\in F_i$. We have
  \begin{align*}
  \dim (xF_j+F_i)S &\leq \dim(xF_jS+F_iS)\\
 &\leq \dim xF_jS +\dim F_iS -\dim (xF_jS\cap F_iS)\\
                   &\leq \dim F_jS +\dim F_iS -\dim S
  \end{align*}
  since $xS\subset (xF_jS\cap F_iS)$. From this we get
  $$\dim (xF_j+F_i)S< \dim F_jS + \dim F_i,$$
since $\dim F_iS \leq \dim F_i + \dim S -1$.
But $F_j\subsetneq F_{j-1}$ implies that $F_jS\subset F_{j-1}S$, whence
$F_jS\subsetneq F_{j-1}S$, since $F_jS=F_{j-1}S$ would contradict $F_j$ being saturated. Now since $F_jS$ and $F_{j-1}S$ are both stabilized by $F_j$ we obtain
  $$\dim F_jS\leq \dim F_{j-1}S -\dim F_{j},$$
  whence
  \begin{equation}
    \label{eq:xFj+Fi}
    \dim (xF_j+F_i)S< \dim F_{j-1}S -\dim F_j +\dim F_i\leq \dim F_{j-1}S,
  \end{equation}
by the hypothesis $\dim F_i\leq \dim F_j$.
To conclude, recall from Lemma~\ref{lem:satdual} {\rm (ii)} that
$(xF_j+F_i)^{**}S=(xF_j+F_i)S$, so that $(xF_j+F_i)^{**}=F_{j-1}$
would contradict \eqref{eq:xFj+Fi}.
\end{proof}

 \begin{lem}\label{claim4}
 $F_i\subset F_1$.
 \end{lem}
\begin{proof}
  By Lemma~\ref{claim1}, there exists $x\in F_i$ such that
  $F_{i-1}x\not\subset F_i$.
  We have
  \begin{equation}
    \label{eq:xF1+Fi}
    \partial (xF_{1}+ F_i)+\partial (xF_{1}\cap F_i)\le \lambda_{1}+\lambda_i.
  \end{equation}
Suppose that $\dim F_i\leq \dim F_1$. Then Lemma~\ref{lem:referee*} implies that $\widetilde{xF_1+F_i}=(xF_1+F_i)^{**}\neq L$, so that
$\partial (xF_1+F_i)\ge \partial (\widetilde{xF_1+F_i})\ge \lambda_1$.
If $\dim F_i\leq \dim F_1$ does not hold, then $\dim xF_1<\dim F_i\leq \dim F_i^*$, and
 Lemma~\ref{lem:sum} implies
 $\partial (xF_1+F_i)\ge \partial (\widetilde{xF_1+F_i})\ge \lambda_1$ again.
In both cases, we obtain from \eqref{eq:xF1+Fi} that
$\partial (xF_1\cap F_i)\le \lambda_i$.
Now $xF_1\cap F_i$ is saturated and contains $x$, but not $F_{i-1}x$
and  not $F_jx$ either for $j\leq i-1$ since $F_j\supset F_{i-1}$.
Since we know that $j$--cells are stabilized by $F_j$ for all $j\leq i-1$,
we obtain that $xF_1\cap F_i$ cannot be a $j$--cell for all $j\leq i-1$.
This implies in particular that 
$\partial(xF_1\cap F_i)\neq \lambda_j$ for all $j\leq i-1$. Hence,
$\partial(xF_1\cap F_i)=\lambda_i$ which implies that $F_i\subset
xF_1$ by minimality of $F_i$ in $\SS_i$.
Since $1\in F_i$ we must have $1\in xF_1$ which implies $xF_1=F_1$.
 \end{proof}

 \begin{prop}\label{claim5}
 For every $j<i$ we have $F_i\subset F_j$.
 \end{prop}

 \begin{proof}
   We prove this by induction on $j$. Lemma~\ref{claim4} gives the result
for $j=1$, so suppose we already have $F_i\subset F_{j-1}$ and let us
prove $F_i\subset F_j$.
Suppose first that $\dim F_i>\dim F_j$.
Let $x\in F_i^*$ and 
consider $Z:=(xF_j+F_i^*)^{**}$, which
belongs to $\SS$ by Lemma \ref{lem:SS} \rm{(v)}. We have, by
Proposition~\ref{prop:subm} and Lemma~\ref{lem:dual},
$$
\partial Z +\partial(xF_j\cap F_i^*)\leq  \partial (xF_j+F_i^*)+\partial (xF_j\cap F_i^*)\le \lambda_j+\lambda_i.
 $$
By Lemma~\ref{lem:sum}, since we assume that $\dim F_i^{**}=\dim F_i\geq
\dim F_j$, we have $Z\neq L$, and by the
induction
hypothesis $F_i\subsetneq F_{j-1}$, we have $F_i^*\supsetneq
F_{j-1}^*$, so that $Z\supsetneq F_{j-1}^*$ and $Z^*\subsetneq F_{j-1}$. 
Therefore $Z^*$ is not a $(j-1)$--cell
by the minimality of $\dim F_{j-1}$ in $\SS_{j-1}$,
from which it is 
not a $k$--cell for $k<j$ in view of $F_1\supset \cdots \supset F_{i-1}$.
Since the dual of a $k$--cell is again a $k$--cell, this implies that
$Z$ is also not a $k$--cell for all $k<j$.
Therefore $\partial Z\geq
\lambda_j$, which implies $\partial (xF_j\cap F_i^*)\le \lambda_i$.
Now, by the hypothesis $\dim F_i > \dim F_j$, we have that $\dim (xF_j\cap
F_i^*)< \dim F_i$ and $xF_j\cap F_i^*$ cannot be an $i$--cell.
Therefore it is in $\SS_k$ for some $k<i$, which, by definition of $i$,
implies that it is stabilized by $F_k$ and hence by $F_{i-1}$.
By applying this to $xF_j$ for every $x\in F_i^*$, we get that the
whole of
$F_i^*$ is stabilized by $F_{i-1}$: but this contradicts
Lemma~\ref{claim1}.
Hence
\begin{equation}
  \label{eq:dimNi<dimNj}
  \dim F_i\leq \dim F_j.
\end{equation}

Next, consider $x\in F_i$. Suppose
that for every $x\in F_i$, $x\neq 0$, $xF_j\cap F_i$ is in $\SS_k$
for some $k<i$. Then every $xF_j\cap F_i$ is stabilized by $F_{i-1}$ and
$F_{i-1}F_i=F_i$ which contradicts Lemma~\ref{claim1}. Therefore there
exists $x\in F_i$ such that $xF_j\cap F_i$ is not in $\SS_k$ for
every $k<i$. 
This choice of $x$ ensures that $\partial (xF_j\cap F_i)\geq \lambda_i$.
If we can show that $\partial (xF_j\cap F_i)= \lambda_i$, we will
conclude
that $F_i\subset xF_j$ by the minimality of the $i$--kernel $F_i$, and since $1\in F_i$ we will have $1\in xF_j$
and $xF_j=F_j$, so that $F_i\subset F_j$ and we will be done.
Consider now
$$
  \partial (xF_j+F_i)+\partial (xF_j\cap F_i)\le \lambda_j+\lambda_i.
$$
This
inequality will yield $\partial (xF_j\cap F_i)\le
\lambda_i$ and the desired result if we can show that
\begin{equation}
  \label{eq:lambdaj}
\partial (xF_j+F_i) \geq \lambda_j.
\end{equation}
Inequality \eqref{eq:lambdaj} will in turn follow
if we show that $(xF_j+F_i)^{**}$ is not
a $k$--cell for $k<j$. We have $xF_j\subset
xF_{j-1}$ and, by the induction
hypothesis on $j$, we have $F_i\subset F_{j-1}$, whence $xF_{j-1}=F_{j-1}$
since $x\in F_i$. Therefore $xF_j+F_i\subset F_{j-1}$. 
From this we derive $(xF_j+F_i)^{**}\subset F_{j-1}^{**}$ and
$(xF_j+F_i)^{**}\subset F_{j-1}$ 
by Lemma~\ref{lem:SS} {\rm (ii)}. Since $F_{j-1}\subsetneq F_k$ for all $k<j-1$, we have, by minimality of $F_k$ in $\SS_k$, that $(xF_j+F_i)^{**}$ cannot be a $k$--cell for $k<j-1$. By Lemma~\ref{lem:referee*} together with \eqref{eq:dimNi<dimNj} we have that $(xF_j+F_i)^{**}$ cannot be a $(j-1)$--cell either
and we are finished. 
 \end{proof}

\begin{prop}\label{claim6}
 $F_i$ is a subfield.
 \end{prop}

 \begin{proof} Let $x\in F_i$, $x\neq 0$; our aim is to show  that $xF_i\subset
   F_i$ (whence $xF_i=F_i$ since $F_i$ is finite dimensional). 
Since $F_i$
is satured it is enough to show that $xF_iS\subset F_iS$. By
contradiction, if $xF_iS\not\subset F_iS$, then there
exists a linear form $\sigma$ such that $\sigma(F_iS)=0$ but
$\sigma(xF_iS)\neq 0$. This last condition translates to $x\notin
F_i^*$ where duality is related to this very choice of non-zero
linear form on $L$.
We would then have $1\in F_i^*$ and $F_i\not\subset F_i^*$. Let us show
now that this is not possible.

For $Z:=(F_i+F_i^*)^{**}$, we have:
$$
  \partial Z +\partial (F_i\cap F_i^*)\le \partial (F_i+F_i^*)+\partial (F_i\cap F_i^*)\le 2\lambda_i.
$$
Since we have proved that $F_i\subset F_j$ for all $j<i$, and these
inclusions are strict,  we have
$F_i^*\supsetneq F_j^*$, whence $Z\supsetneq F_j^*$. Note that
$Z\in\SS$ by Lemma~\ref{lem:SS} {\rm (v)}.
Since $F_j^*$ is a $j$--cell whose dual has minimum
dimension, $Z$ cannot be a $j$--cell for $1\leq j<i$, otherwise
$Z^*$ would also be a $j$--cell, and $Z^*\subsetneq F_j$ would contradict
the minimality of $F_j$ in $\SS_j$.
By Lemma~\ref{lem:sum} we also have
$Z\neq L$, so we conclude that
$$\partial Z \geq \lambda_i.$$
Hence $\partial (F_i\cap F_i^*)\leq \lambda_i$, which implies
$F_i\subset F_i^*$ since $F_i$ is an $i$--kernel and has smaller dimension than any
$j$--cell for $j<i$ by Proposition~\ref{claim5}. 
 \end{proof}

 We make the passing remark that the proof of Proposition~\ref{claim6}
 exploits the fact that many different linear forms $\sigma$ can be used
 to define duality: this breaks significantly from the additive setting
 where combinatorial duality is achieved through complementation and
 can therefore be defined only in a unique way.
 
\begin{prop}\label{claim7}
 Every $i$--cell is stabilized by $F_i$.
\end{prop}

 \begin{proof}
Let us suppose that there exists an $i$--cell $X$  that is not stabilized
by $F_i$ and work towards a contradiction.
Without loss of generality we may assume that $X$ is of finite
dimension by Corollary~\ref{cor:subfield}.

The proof strategy consists in
constructing smaller and smaller $i$--cells that are not
stabilized by $F_i$ until we eventually exhibit one that is included
in $xF_i$ for some $x$, which will yield a contradiction.

That $X$ is not stabilized by $F_i$ 
means there exists $x\in X$ with $xF_i\not\subset X$. 

We first argue that there exists $k$, $1\leq k\leq i-1$, such that
$X\cap xF_k$ is an $i$--cell not stabilized by $F_i$.

We have 
$$\partial(X+xF_i) + \partial(X\cap xF_i)\leq 2\lambda_i.$$
Since $X\cap xF_i\subsetneq xF_i$ we have $\partial (X\cap xF_i)>
\lambda_i$. This is because $X\cap xF_i$ is a saturated set whose dimension is smaller than that of any $j$-cell for $1\leq j\leq i$ by Proposition~\ref{claim5}.
Therefore $\partial(X+xF_i)<\lambda_i$. Furthermore,
by Lemma \ref{lem:sum}, $(X+xF_i)^{**}\neq L$, so that
$\partial(X+xF_i)^{**} = \lambda_k$ for some $1\leq k\leq i-1$.
Now, since $(X+xF_i)^{**}\in \SS_k$ (by Lemma~\ref{lem:SS} {\rm (v)}), and $k<i$, we know that
$(X+xF_i)^{**}=\widetilde{X+xF_i}$ is stabilized by $F_k$, whence
$$(X+xF_i)^{**} = (X+xF_k)^{**}$$
and $\partial (X+xF_k)^{**}=\lambda_k$.
We now write
$$\partial (X+xF_k)^{**} + \partial(X\cap xF_k)\leq
\lambda_k+\lambda_i$$
from which we get $\partial (X\cap xF_k)\leq \lambda_i$,
which implies $\partial (X\cap xF_k) = \lambda_i$, since
otherwise $X\cap xF_k$ is an $\ell$--cell for some $\ell <i$, and therefore
stabilized by $F_\ell$, and hence by $F_i$, which contradicts our
assumption on $x$.

The space $X\cap xF_k$ is therefore an $i$-cell that is not stabilized by $F_i$,
and we may therefore replace $X$ by an $i$-cell which is included in some kernel.
Specifically, let $j\leq i$ be the largest integer such that
there exists an $i$--cell $X$ not stabilized by $F_i$ and included in
$xF_j$ for some~$x\in X$ with $xF_i\not\subset X$. Clearly we can only have $j\leq i-1$ since there
is no $i$--cell  included in but not equal to $xF_i$.
We have just shown $j\geq 1$. 
Furthermore, we have also shown that
$$\partial(X+xF_i)^{**} =\lambda_k$$
with $k\leq j$; otherwise, repeating the above argument, there would exist a new $i$--cell of the form
$X\cap xF_k$ that is not stabilized by $F_i$, which would contradict our definition of $j$.

Since $X+xF_i \subset xF_j$ and $xF_j$ is
saturated, we cannot have $(X+xF_i)^{**}\in \SS_k$ for $k<j$ as this would
imply $(X+xF_i)^{**}=\widetilde{X+xF_i}\subset xF_j$, meaning $xF_j$ contains a $k$--cell, which is not possible as a $j$--kernel is strictly smaller than a $k$--cell for $k<j$ in view of Proposition~\ref{claim5} and the minimality of $\dim F_k$ over all $k$--cells.
Therefore, we have just shown that 
\begin{equation}
  \label{eq:inSj}
  (X+xF_i)^{**}\in \SS_j.
\end{equation}

Our next objective is to construct an $i$--cell that is not stabilized
by $F_i$ and included in a $(j+1)$--kernel $yF_{j+1}$, which will contradict the
definition of $j$ and prove the proposition. For this we will need to
apply Lemma~\ref{lem:sum} to the space $X^*$ and to the
$(j+1)$--kernel, for which we need the condition 
\begin{equation}
  \label{eq:dimX}
  \dim X \geq \dim
F_{j+1}
\end{equation}
 which we now prove. We assume $j<i-1$, since
 if $j=i-1$ \eqref{eq:dimX} is immediate as $X$ is an $i$--cell.
 
From $F_{j+1}S\subsetneq F_jS$ (the $F_j$ are saturated sets) we have,
since the $F_j$ are subfields,
\begin{equation}
  \label{eq:FjS}
  \dim F_jS \geq \dim F_{j+1}S+\dim F_{j+1}.
\end{equation}
For the same reason, since $F_iS\subsetneq F_{j+1}S$, we have
\begin{equation}
  \label{eq:Fj+1S}
  \dim F_{j+1}S \geq \dim F_iS+\dim F_i.
\end{equation}
Now from
$$\partial(X+xF_i) + \partial(X\cap xF_i) \leq 2\lambda_i$$
we have $\partial(X+xF_i)\leq \lambda_i$, i.e.,
$$\dim(X+xF_i)S - \dim(X+xF_i) \leq \lambda_i$$
meaning
\begin{align}
  \dim(X+xF_i)S &\leq \lambda_i + \dim(X+xF_i)\nonumber\\
                &< \lambda_i + \dim X + \dim F_i.\label{eq:lambdai}
\end{align}
From \eqref{eq:inSj}, $X\subset xF_j$, and the fact that
$j$--cells are stabilized by $F_j$, 
we have $(X+xF_i)S = (X+xF_j)S = xF_jS$. Writing $\lambda_i = \dim
F_iS-\dim F_i$, we get from \eqref{eq:lambdai}
\begin{align*}
  \dim F_jS =\dim(X+xF_i) &< \dim F_iS - \dim F_i +\dim X + \dim F_i\\
            &\leq \dim F_iS +\dim X\\
            &< \dim S + \dim F_i + \dim X.
\end{align*}
But on the other hand, \eqref{eq:FjS} and \eqref{eq:Fj+1S} imply
$$\dim F_jS \geq \dim F_iS + \dim F_i + \dim F_{j+1}\geq \dim S +\dim
F_i + \dim F_{j+1},$$
whence $\dim F_{j+1} < \dim X$.

Now that we have proved \eqref{eq:dimX}, we
are ready to construct an $i$--cell that is not stabilized
by $F_i$ and included in a $(j+1)$--kernel $yF_{j+1}$.

Since $X$ is assumed not to be stabilized by $F_i$, $X^*$ is not
stabilized by $F_i$ either by Corollary~\ref{cor:subfield}. Therefore
there exists $y\in X^*$ such that $yF_i\not\subset X^*$. We write
\begin{equation}
  \label{eq:X*andFj+1}
  \partial(X^*+yF_{j+1}) + \partial(X^*\cap yF_{j+1})\leq
\lambda_{j+1}+\lambda_i.
\end{equation}
By the hypothesis $X\subsetneq xF_j$, we have 
$X^*+yF_{j+1}\supset X^*\supsetneq (xF_j)^*$, and
$(X^*+yF_{j+1})^*\subsetneq xF_j$. Therefore,
$(X^*+yF_{j+1})^*$ and $(X^*+yF_{j+1})^{**}$ do not belong
to $\SS_k$
for any $1\leq k\leq j$. Now \eqref{eq:dimX}
 and Lemma~\ref{lem:sum} imply that
$(X^*+yF_{j+1})^{**}\neq L$. Therefore,
$$\partial(X^*+yF_{j+1})\geq \lambda_{j+1}.$$
Together with \eqref{eq:X*andFj+1} this implies
$$\partial(X^*\cap yF_{j+1}) \leq \lambda_i.$$
If $X^*\cap yF_{j+1}$ were a $k$--cell for $k<i$, it would be
stabilized by $F_k$ and hence $F_i$: since this is assumed not to be
the case, we have that $X^*\cap yF_{j+1}$ must be an $i$--cell.
As announced, we have constructed an $i$--cell that is included
in a $(j+1)$--kernel $yF_{j+1}$, which contradicts the definition of $j$.
 \end{proof}

Propositions \ref{claim5}, \ref{claim6} and \ref{claim7} imply that $i=n$
and
prove Theorem~\ref{thm:linbal}.

\section{Proof of Theorem~\ref{thm:kneser_onesided}}\label{sec:proofmain}
By replacing $S$ if need be by $s^{-1}S$ for some $s\in S$, we may always suppose $1\in S$. 
Suppose first $L=F(S)$.
If $T=F(S)$ is the only saturated subspace of finite dimension such that 
\begin{equation}
  \label{eq:ST}
  \dim ST <\dim S+\dim T -1,
\end{equation}
then the theorem holds with $K=F(S)$. Otherwise,
if there exists a finite-dimensional subspace $T$ satisfying
\eqref{eq:ST} and such that $ST\neq F(S)$, 
then we are in the conditions of Theorem~\ref{thm:linbal}. 
In this case $\tilde{T}$ is a $k$--cell for some $1\leq k\leq n-1$, and
in particular is stabilized by $F_k$, implying that
 $ST=S\tilde{T}$ is stabilised by $F_k$ as well. Since $F(S)$ and also
every space $F_k$ contain $F_{n-1}$, the conclusion of the theorem holds with $K=F_{n-1}$.

Consider now the case $L\supsetneq F(S)$. 
Let $T$ be a subspace satisfying \eqref{eq:ST}.

  Let $t\in T, t\neq 0$. Let $T_t= tF(S)\cap T$: since $T_t$ is
an $F$-vector space, we may write $T=T_t\oplus T'$ for some subspace
$T'$ of $T$ with $T'\cap T_t=\{0\}$. 
Note that this implies $T'\cap tF(S)=\{0\}$.
Since $ST_t\subset tF(S)$, we
have $ST_t\cap T'=\{0\}$, and $ST\supset ST_t \oplus T'$ implies
$$\dim ST \geq \dim ST_t +\dim T'.$$
From 
$$\dim S + \dim T_t+\dim T' -1 = \dim S +\dim T-1 > \dim ST,$$
we get
$$\dim S + \dim T_t -1 > \dim ST_t.$$
Now we get from the case $L=F(S)$ the existence of a subspace $K$ such
that $ST'K=ST'$ for any subspace $T'$ satisfying \eqref{eq:ST} and
included in $F(S)$, or in a $1$-dimensional $F(S)$-vector space.
Therefore we have $ST_tK\subset ST$ for every $t$ which concludes the
proof of Theorem~\ref{thm:kneser_onesided}.


\end{document}